\documentclass[12pt]{article}
\usepackage{amsmath}
\usepackage{amsfonts}
\usepackage{pb-diagram}
\usepackage{cancel}
\usepackage{amsthm} 
\usepackage{amscd}
\usepackage{amssymb}
\usepackage{mathdots}
\newcounter{numb}

\theoremstyle{plain} 
\newtheorem{theorem}{Theorem}[section]
\newtheorem{corollary}[theorem]{Corollary}

\newtheorem{lemma}[theorem]{Lemma}

\newtheorem{remark}[theorem]{Remark}
\numberwithin{equation}{section}
\newcommand{\ds}{\displaystyle}
\newcommand{\al}{\alpha}

\newcommand{\x}{\mathcal{X}}

\newcommand{\w}{\wedge}

\newcommand{\p} {\partial}
\newcommand{\g} {\mathfrak {gal}}
\newcommand{\f} {\mathfrak {g}}

\newcommand{\sln}{\mathfrak {sl}}
\newcommand{\gln}{\mathfrak {gl}}
\newcommand{\shn}{\mathfrak {sch}}

\newcommand{\h}{\mathfrak {h}}
\newcommand{\I}{\mathfrak {I}}
\newcommand{\A}{\mathfrak {A}}

\newcommand{\so}{\mathfrak {so}}
\begin{document}
\title{On the Schr\"{o}dinger group}
\author{Guy Roger Biyogmam}
\date{}
\maketitle{.\\Department of Mathematics,\\Southwestern Oklahoma State University,\\Weatherford, OK 73096, USA\\ \textit{Email:}\textit{guy.biyogmam@swosu.edu}}
\begin{abstract}

In this paper, we compute the Leibniz homology of  the Schr\"{o}dinger algebra.  We show that it is a graded vector space generated by  tensors in dimensions  $2n-2$ and $2n$. The Leibniz homology of the full Galilei algebra is also calculated. 
\end{abstract}
\section{Introduction}
Leibniz homology was introduced by Jean-Louis Loday (see \cite[10.6]{L}) as a non commutative version of Lie algebra homology. In this paper, we calculate this homology for one of the most important non semisimple Lie algebra of mathematical physics. Thanks to its  semidirect sum structure, the Schr\"{o}dinger  algebra is presented as  an abelian extension of a semisimple Lie algebra. This enables us to calculate its Leibniz (co)homology using techniques previously applied on the Lie algebra of the euclidean group \cite{BGR}, the affine symplectic Lie algebra \cite{JLM} and  the Poincar\'e algebra \cite{RGB}. 
Recall that the (non centrally extended or massless) full Galilei group $\widetilde{GAL}(n)$ of $n$-dimensional space consists of real $(n+2)\times (n+2)$ matrices 
$$\begin{bmatrix}  \x & v & a \\ 0 & A_n  & B_n \\ 0 & C_n & D_n \end{bmatrix}~~~~~~~~~~~~~~~~~~~(M.1)$$

with $\x\in O(n;~\mathbb{R}),$ $a, v\in \mathbb{R}^n$ and  $A_n,B_n,C_n,D_n \in \mathbb{R}$ with $A_nD_n-B_nC_n\neq 0.$ Its group structure is the semidirect product $$\widetilde{GAL}(n)=(O(n;\mathbb{R})\times GL(2;\mathbb{R}))\ltimes(\mathbb{R}^n\times\mathbb{R}^n).$$

 With the condition $A_nD_n-B_nC_n=1,$ the matrices  $(M.1)$ above constitute the Schr\"{o}dinger group   $Sch(n).$  Its Lie algebra $\shn_n$  is an abelian extension of the Lie algebra $$\bar{\h}_n=\so(n; \mathbb{R})\oplus\sln(2; \mathbb{R}).$$ We  calculate its Leibniz homology and obtain  the isomorphism of  graded vector spaces
 $$HL_*(\shn_n;~\mathbb{R})\cong \big(\mathbb{R}\oplus
\langle \tilde{\zeta}_{n}\rangle\oplus\left\langle \tilde{\alpha}_{n}\right\rangle\big)\otimes T^*(\tilde{\gamma}_n),$$
 where  
 $\langle \tilde{\zeta}_{n}\rangle$  denotes a 1-dimensional vector space in dimension $2n-2$ generated
by the $\shn_n$-invariant 
 $$
\begin{aligned}\tilde{\zeta}_n &=\frac{1}{(2n)!}\sum_{{\substack{\\\sigma\in S_{2n-2} }}}\mbox{sgn} (\sigma)y_{\sigma(1)}\otimes\ldots\widehat{y_{\sigma(i)}}\ldots\otimes y_{\sigma(n)}\otimes y_{\sigma(n+1)}\otimes\ldots\widehat{y_{\sigma(n+i)}}\ldots\otimes y_{\sigma(2n)}
\end{aligned},$$  
 $\left\langle \tilde{\al}_{n}\right\rangle$  denotes a 1-dimensional vector space in dimension $2n$ generated
by the $\shn_n$-invariant 
 $$
\begin{aligned}\tilde{\alpha}_n &=\frac{1}{(2n)!}\sum_{\sigma\in S_{2n}}\mbox{sgn} (\sigma)y_{\sigma(1)}\otimes\ldots\otimes y_{\sigma(n)}\otimes y_{\sigma(n+1)}\otimes\ldots\otimes y_{\sigma(2n)}
\end{aligned}$$  and $T^*(\tilde{\gamma}_n)$ denotes the tensor algebra on a $(2n-2)$-degree generator which is an antisymmetrization  of  $$
\begin{aligned}\gamma_n &=\sum_{{\substack{1\leq i<j\leq n \\ }}} X_{ij}\otimes( y_{1}\w\ldots\widehat{y_i}\ldots\w y_{n}\w y_{n+1}\w\ldots\widehat{y_{n+j}}\ldots\w y_{2n})\\&~~~~-\sum_{{\substack{1\leq i<j\leq n \\ }}} X_{ij}\otimes( y_{1}\w\ldots\widehat{y_j}\ldots\w y_{n}\w y_{n+1}\w\ldots\widehat{y_{n+i}}\ldots\w y_{2n})
\end{aligned}$$  
with 
$\ds{y_i:=x_i\frac{\p}{\p x^{n+1}}~~\mbox{and} ~~y_{n+i}:=x_i\frac{\p}{\p x^{n+2}}~~ \mbox{for}~~1\leq i\leq n}.$

 
Since $\gln(2;~\mathbb{R})\cong\sln(2;\mathbb{R})\oplus D$ where $\ds{D=\big\{\begin{bmatrix}  d  & 0 \\ 0 & d \end{bmatrix}, d\in \mathbb{R} \big\}},$ it follows that $\ds{\widetilde{\g}_n\cong \shn_n\oplus D.}$ The Leibniz homology of $\widetilde{\g}_n$ is then obtained via a K\"unneth-style formula for the homology of Leibniz algebras \cite{LK}.

\section{The Leibniz Homology}
Recall that for any Lie algebra $\f$ over a ring $k$ and $V$ any $\f$-module, the Lie algebra homology of $\f$ with coefficients in the module $V,$ written $H^{Lie}_*(\f;~V),$ is the homology of the Chevalley-Eilenberg complex $V\otimes \wedge^*(\f),$ namely $$V\stackrel{d}{\longleftarrow} V\otimes \f^{\wedge^1}\stackrel{d}{\longleftarrow} V\otimes \f^{\wedge^2}\stackrel{d}{\longleftarrow}\ldots\stackrel{d}{\longleftarrow} V\otimes\f^{\wedge^{n-1}}\stackrel{d}{\longleftarrow} V\otimes\f^{\wedge^n}\leftarrow\ldots$$ where $\f^{\wedge^n}$ is the $n$th exterior power of $\f$ over $k,$ and where 
$$d(v\otimes g_1\w\ldots\wedge g_n)=\sum_{1\leq j\leq n}(-1)^j[v,~g_j]\otimes g_1\wedge\ldots \hat{g_j}\ldots\wedge g_n$$ $$ +\sum_{1\leq i<j\leq n}(-1)^{i+j-1}v\otimes [g_i,~g_j]\wedge g_1\wedge\ldots \hat{g_i}\ldots\hat{g_j}\ldots\wedge g_n~~~~~\cite{CE}$$ where $\hat{g}_i$ means that the variable $g_i$ is deleted.

For each $n,$ we have the canonical projection $\pi:\f\otimes \f^{\w n}\longrightarrow \f^{\w n+1}.$ This gives a map of chain complexes $\f\otimes\wedge^*(\f)\longrightarrow\wedge^{*+1}(\f)$ and thus induces a $k$-linear map on homology $$\pi_*:H^{Lie}_*(\f;~\f)\longrightarrow H^{Lie}_{*+1}(\f;~k).$$  
Recall that for any  $\f$-module $M,$ the submodule  $M^{\f}$ of $\f-$invariants is defined by $$M^{\f}=\left\{ m\in M~|~ [m,~g]=0 ~ \mbox{for all}~g\in \f\right\}.$$ 

\begin{lemma}\label{lem0}
Let  $$0\longrightarrow \I\stackrel{i}{\longrightarrow} \f_e\stackrel{\pi}{\longrightarrow} \f\longrightarrow 0$$ be an abelian extension of a (semi)-simple Lie algebra $\f$ over $\mathbb{R}.$ Then the following are natural vector space isomorphisms
 $$H_*^{Lie}( \f_e;~\mathbb{R})\cong H_*^{Lie}(\f;~\mathbb{R})\otimes [\wedge^*(\I)]^{\f},$$ $$H_*^{Lie}( \f_e;~ \f_e)\cong H_*^{Lie}(\f;~\mathbb{R})\otimes [ H^{Lie}_*(\I;~ \f_e)]^{\f}.$$
\end{lemma}
\begin{proof}
The proof consists of applying the Hochschild-Serre spectral sequence to the (semi)-simple Lie algebra $\f,$ subalgebra of $  \f_e.$ See \cite[lemma 2.1]{JL} for details when $\f$ is simple.
\end{proof}

\begin{remark}
The (co)homology groups of $\sln(2;~\mathbb{R})$  and $\so(n;~\mathbb{R})$ are known (see \cite[p.1742]{IK}).
So by lemma \ref{lem0}, to determine $H_*^{Lie}( \shn_n;~\mathbb{R})$ and $H_*^{Lie}( \shn_n;~\shn_n),$ it is enough to determine the appropriate modules of $\f$-invariants.
\end{remark}

Recall also that for a Leibniz algebra ( Lie algebra in particular) $\f,$  the Leibniz homology of $\f$ with coefficients in $\mathbb{R}$ denoted $HL_*(\f,\mathbb{R}),$ is the homology of the Loday complex $T^*(\f),$ namely $$\mathbb{R}\stackrel{0}{\longleftarrow}\f\stackrel{[~,~]}{\longleftarrow}\f^{\otimes^2}\stackrel{d}{\longleftarrow}\ldots\stackrel{d}{\longleftarrow}\f^{\otimes^{n-1}}\stackrel{d}{\longleftarrow}\f^{\otimes^n}\leftarrow\ldots$$ where $\f^{\otimes^n}$ is the $n$th tensor power of $\f$ over $\mathbb{R},$ and where
 $$
\begin{aligned} d(g_1\otimes g_2\otimes\ldots\otimes g_n)=&\\ &\sum_{1\leq i<j\leq n}(-1)^{j}g_1\otimes g_2\otimes\ldots\otimes g_{i-1}\otimes [g_i,g_j]\otimes g_{i+1}\otimes\ldots \widehat{g_j}\ldots\otimes g_n~~~~~ \cite{L}.
\end{aligned}$$ 
 Note that the latter complex above is infinite, so the calculation of these homology groups is most likely possible only through the use of a spectral sequence. Pirashvili \cite{TP} introduced a spectral sequence for this purpose and Lodder \cite{JL} used it to establish a structure theorem useful to determine the Leibniz homology groups of abelian extensions of (semi-)simple Lie algebras in terms of these Lie algebras invariants. 

\section{Leibniz homology of the Schr\"{o}dinger algebra}

Assume that $\mathbb{R}^n$ is given the coordinates $(x_1,x_2,...,x_n),$ and let $\frac{\partial}{\partial x^i}$ be the unit vector fields parallel to the $x_i$ axes respectively. It is easy to show that the Lie algebra generated by the family $\bar{B_1}$ below of vector fields (endowed with the bracket of vector fields) is  isomorphic to $\bar{\h}_n$: $$\bar{B}_1=\{X_{ij},a_n,b_n,c_n\}$$ where  \\$\ds{X_{ij}:=-x_i\frac{\partial}{\partial x^j}+x_j\frac{\partial}{\partial x^i}~~~1\leq i< j\leq n},$ \\ 
$\ds{a_n:=-x_{n+1}\frac{\partial}{\partial x^{n+1}}+x_{n+2}\frac{\partial}{\partial x^{n+2}}},$\\
$\ds{b_n:=x_{n+1}\frac{\partial}{\partial x^{n+2}}},$\\$\ds{c_n:=-x_{n+2}\frac{\partial}{\partial x^{n+1}}}.$\\
The brackets relations of the Lie algebra $\bar{\h}_n$ are:\\

$\ds{[X_{ij},X_{ik}]=X_{jk},~~~~[X_{ij},a_{n}]=0,~~~~[X_{ij},b_{n}]=0,~~~~[X_{ij},c_{n}]=0,}$\\

$\ds{[a_{n},b_{n}]=-2b_{n},~~~~[a_{n},c_{n}]=2c_n,~~~~[b_{n},c_{n}]=a_{n}}.$

\begin{remark} The brackets above yield the following  Lie algebra isomorphisms \\$\ds{\sln(2; \mathbb{R})\cong span\{a_{n},b_{n},c_{n}\}},$\\ $\ds{\so(n; \mathbb{R})\cong span\{X_{ij}, ~1\leq i<j\leq n\}}.$
\end{remark}

Denote by $\I_n$ the Lie algebra of $\mathbb{R}^n\times \mathbb{R}^n.$  The following is a vector space basis of $\I_n.$ $$B_2=\{x_i\frac{\partial}{\partial x^{n+1}}, ~x_i\frac{\partial}{\partial x^{n+2}}~~1\leq i\leq n\}.$$ Then the Schr\"{o}dinger algebra $\shn_n$ has an $\mathbb{R}$-vector space basis $\bar{B}_1\cup B_2$  and there is a short exact sequence of Lie algebras \cite[p.203]{KA} 
 $$0\longrightarrow \I_n\stackrel{i}{\longrightarrow} \shn_n\stackrel{\pi}{\longrightarrow} \bar{\h}_n\longrightarrow 0$$ where $i$ is the inclusion map and $\pi$ is the projection $~~\ds{\shn_n\longrightarrow (\shn_n/\I_n)\cong\bar{\h}_n.}$    
 Note that here, $\I_n$ is the standard representation of $\bar{\h}_n$ i.e $\bar{\h}_n$ acts on $\I_n$ via matrix multiplication on vectors. More precisely,
 $\I_n$ is an abelian ideal of $\shn_n$ acting on $\shn_n$  via the following brackets of vector fields:\\

 $\ds{[X_{ij},x_i\frac{\partial}{\partial x^{n+1}}]=x_j\frac{\partial}{\partial x^{n+1}},~~~~[X_{ij},x_i\frac{\partial}{\partial x^{n+2}}]=x_j\frac{\partial}{\partial x^{n+2}},~~~~[a_{n},x_i\frac{\partial}{\partial x^{n+1}}]=x_i\frac{\partial}{\partial x^{n+1}}},$\\

  $\ds{[a_{n},x_i\frac{\partial}{\partial x^{n+2}}]=-x_i\frac{\partial}{\partial x^{n+2}}, ~~~~[b_{n},x_i\frac{\partial}{\partial x^{n+1}}]=-x_i\frac{\partial}{\partial x^{n+2}},~~~~[b_{n},x_i\frac{\partial}{\partial x^{n+2}}]=0,}$\\

$\ds{[c_{n},x_i\frac{\partial}{\partial x^{n+1}}]=0,~~~~[c_{n},x_i\frac{\partial}{\partial x^{n+2}}]=x_i\frac{\partial}{\partial x^{n+1}},~~~~[x_i\frac{\partial}{\partial x^{n+1}},x_j\frac{\partial}{\partial x^{n+2}}]=0}.$\\

In this framework, $\ds{X_{ij},a_n,b_n,c_n,x_i\frac{\p}{\p x^{n+1}}}$ and  $\ds{x_i\frac{\p}{\p x^{n+2}}}$ represent respectively the generators of the rotations, dilation, time translation (Hamiltonian), conformal transformation, Galilean boosts and
space translations (momentum operators).

Also  the Lie algebra $\bar{\h}_n$ acts  on $\I_n$ and $\shn_n$  via the bracket of vector fields. This action is extended to $\I_n^{\wedge k}$ by $$[\al_1\w\al_2\w\ldots\w\al_k,~X]=\sum^k_{i=1}\al_1\w\al_2\w\ldots\w [\al_i,~X]\w\ldots\w\al_k$$ for $\al_i\in\I_n,~X\in\bar{\h}_n,$ and the action of $\bar{\h}_n$ on $\shn_n\otimes\I_n^{\w k}$ is given by  $$[g\otimes\al_1\w\al_2\w\ldots\w\al_k,~X]=[g,~X]\otimes\al_1\w\ldots\w\al_k$$$$+\sum^k_{i=1}g\otimes\al_1\w\al_2\w\ldots\w [\al_i,~X]\w\ldots\w\al_k$$ for $g\in\shn_n.$

For the remaining of the paper, we write $\so(n),$  $\gln(2)$ and $\sln(2)$ for $\so(n; \mathbb{R}),$  $\gln(2; \mathbb{R})$  and  $\sln(2; \mathbb{R})$ respectively.

\begin{lemma}\label{lem2}
There is a graded vector space isomorphism:\\
$$[\w^*(\I_n)]^{\bar{\h}_n}\cong\mathbb{R}\oplus\langle\beta_n\rangle\oplus\langle\zeta_n\rangle\oplus\langle\alpha_n\rangle$$ where
$$\alpha_n=y_{1}\w\ldots\w y_{n}\w y_{n+1}\w\ldots\w y_{2n}~,~~~\beta_n=\sum_{i=1}^n y_i\w y_{n+i}$$ and $$\zeta_n=\sum_{i=1}^n y_1\w\ldots\widehat{y_i}\ldots\w y_n\w y_{n+1}\ldots \widehat{ y_{n+i}}\ldots\w y_{2n}$$
with
$\ds{y_i:=x_i\frac{\p}{\p x^{n+1}}~~\mbox{and} ~~y_{n+i}:=x_i\frac{\p}{\p x^{n+2}}~~ \mbox{for}~~1\leq i\leq n}.$

\end{lemma}

\begin{proof}
Clearly $[\mathbb{R}]^{\bar{\h}_n}=\mathbb{R}.$ Now let $\omega\in\I_n. $ Then  $\ds{\omega=\sum_{1\leq i\leq n }c_ix_i\frac{\partial}{\partial x^{n+1}}+\sum_{1\leq i\leq n }c'_ix_i\frac{\partial}{\partial x^{n+2}}}$ for real constants $c_i,c'_i.$ Assume without loss of generality that $c_{i_0}\neq 0$ for some $i_0\neq n.$ Then $\ds{[\omega,X_{i_0n}]=-c_{i_0}x_n\frac{\partial}{\partial x^{n+1}}+c_nx_{i_0}\frac{\partial}{\partial x^{n+1}}\neq 0}.$ So $\omega\notin [\I_n]^{\bar{\h}_n}$ and thus $ [\I_n]^{\bar{\h}_n}=0.$ Now write $\ds{\I_n=\I^1_n\oplus\I^2_n}$ with $\ds{\I_n^1=\big<y_1,\ldots,y_n\big>}$ and $\ds{\I^2_n=\big<y_{n+1},\ldots,y_{2n}\big>}.$ Then since $\so(n)$ is a Lie subalgebra of $\bar{\h}_n,$ it follows by \cite[lemma 4.1]{BGR} that  $$\big[\w^k(\I^i_n)\big]^{\bar{\h}_n}\subseteq\big [\w^k(\I^i_n)\big]^{\so(n)}=\begin{cases}\mathbb{R} & \mbox{if} ~~k=0\mbox{}\\ \big\langle x_1\frac{\p}{\p x^{n+i}}\w\ldots\w x_n\frac{\p}{\p x^{n+i}}\big\rangle& \mbox{if}~~k=n\mbox{}\\0 & \mbox{else}~~\mbox{}\end{cases}$$
 for $i=1,2.$ However,  $\ds{x_1\frac{\p}{\p x^{n+i}}\w\ldots\w x_n\frac{\p}{\p x^{n+i}}\notin [\w^k(\I^i_n)]^{\bar{\h}}}$ because $$\big [x_1\frac{\p}{\p x^{n+i}}\w\ldots\w x_n\frac{\p}{\p x^{n+i}},~a_n\big]=(-1)^{i}n~x_1\frac{\p}{\p x^{n+i}}\w\ldots\w x_n\frac{\p}{\p x^{n+i}}\neq 0.$$
 Now let $\ds{\omega=\sum_{A_*,B_*}c_{**}~ A_*\w B_*~\in \big[(\I^1_n)^{\w r}\w (\I^2_n)^{\w s} \big]^{\bar{\h}_n}}$ with $A_*\in (\I^1_n)^{\w r}~\mbox{and}~B_*\in (\I^2_n)^{\w s}.$ \\If $r\neq s,$   
we have $\ds{[\omega, a_n]=\big[\sum_{A_*,B_*}c_{**}~ A_*\w B_*,~a_n\big]=(s-r)\omega\neq 0}.$ This is a contradiction.\\ 
If $r=s=1,$ one easily shows that $$\ds{ [(\I^1_n)\w (\I^2_n) ]^{{\bar{\h}_n}}=\langle\beta_n\rangle}.$$ If $r=s=n-1,$ one also  shows  that $$\ds{ [(\I^1_{n-1})^{n-1}\w (\I^2_{n-1})^{n-1} ]^{{\bar{\h}_n}}=\langle\zeta_n\rangle}.$$ 
If $r=s=n,$ a straightforward calculation shows  that $$\ds{ [(\I^1_n)^{\w n}\w (\I^2_n)^{\w n} ]^{\bar{\h}_n}=\langle\al_n\rangle}.$$ For $\ds{1<r=s<n-1,}$ we show  that $ \ds{[(\I^1_n)^{\w s}\w (\I^2_n)^{\w s} ]^{\bar{\h}_n} =0}$ by showing by induction on $n$ that  $\ds{ [(\I^1_n)^{\w s}\w (\I^2_n)^{\w s} ]^{{\so(n)}}=0.}$ Indeed, it is easy to check the result for $n=4$ and $r=s=2.$ By the inductive hypothesis, suppose $\ds{ [(\I^1_{n-1})^{\w s}\w (\I^2_{n-1})^{\w s} ]^{{\so(n-1)}}=0}$ for $s\neq 0,~1,~n-2,~n-1$ and let $\ds{z\in [(\I^1_n)^{\w s}\w (\I^2_n)^{\w s} ]^{{\so(n)}}}$ with $s \neq 0,~1,~n-1,~n$ fixed. Then  $${z=c_1A_1\w B_1+c_2 A_2\w B_2\w   y_n+c_3A_3\w B_3\w y_{2n}+c_4A_4\w B_4\w y_n\w y_{2n}}$$ where $A_1,A_3\in(\I^1_{n-1})^{\w s},$ $B_1,B_2\in(\I^2_{n-1})^{\w s},$ $A_2,A_4\in(\I^1_{n-1})^{\w s-1},$ $B_3,B_4\in(\I^2_{n-1})^{\w s-1},$ and $c_1,c_2,c_3,c_4\in\mathbb{R}.$ 
Let $X\in\so(n-1)\subseteq\so(n)$ as a Lie subalgebra, we have $$0=[z,~X]=c_1[A_1\w B_1,~X]+c_2[A_2\w B_2,~X]\w y_n+c_3[A_3\w B_3,~X]\w y_{2n}+c_4[A_4\w B_4,~X]\w y_n\w y_{2n}.$$ If the coefficients are non-zero, then  the terms $[A_1\w B_1,~X],~$ $[A_2\w B_2,~X],~$ $[A_3\w B_3,~X]$ and $[A_4\w B_4,~X]$ are all zero By linear independence.  This implies that $$\ds{A_2\w B_2\in[(\I^1_{n-1})^{\w s-1}\w (\I^2_{n-1})^{\w s} ]^{{\so(n-1)}}=0},~~ \ds{A_3\w B_3\in[(\I^1_{n-1})^{\w s}\w (\I^2_{n-1})^{\w s-1} ]^{{\so(n-1)}}=0}$$ by the case $r\neq s$ above, and $$\ds{A_1\w B_1\in[(\I^1_{n-1})^{\w s}\w (\I^2_{n-1})^{\w s} ]^{{\so(n-1)}}=0,}~~\ds{A_4\w B_4\in[(\I^1_{n-1})^{\w s-1}\w (\I^2_{n-1})^{\w s-1} ]^{{\so(n-1)}}=0}$$ by inductive hypothesis. Hence $z=0$.
\end{proof}

\begin{lemma}\label{le3}  For all integer $k,$
$$[\sln(2)\otimes\I_n^{\w k}]^{\bar{\h}_n}=0$$ 
\end{lemma}
\begin{proof}
Let $k_1$ and $k_2$ with $0\leq k_1,k_2\leq n$ and let $\ds{\omega\in [\sln(2)\otimes(\I^1_n)^{\w k_1}\w(\I^2_n)^{\w k_2}]^{\bar{\h}}}.$ Then $$\omega=\sum_{A_*,B_*}c_{a}^{**}~a_n\otimes A_*\w B_*+\sum_{A_*,B_*}c_{b}^{**}~b_n\otimes A_*\w B_*+\sum_{A_*,B_*}c_{c}^{**}~c_n\otimes A_*\w B_*$$ where $ A_*\in(\I^1_n)^{\w k_1},~B_*\in(\I^2_n)^{\w k_2},$ and $c_{a}^{**},c_{b}^{**},c_{c}^{**}$ are real coefficients.  
$$
\begin{aligned}
0=[\omega,a_n]=&(k_2-k_1)\sum_{A_*,B_*}c_{a}^{**}~a_n\otimes A_*\w B_*\\&+(k_2-k_1+2)\sum_{A_*,B_*}c_{b}^{**}~b_n\otimes A_*\w B_*+(k_2-k_1-2)\sum_{A_*,B_*}c_{c}^{**}~c_n\otimes A_*\w B_*.
\end{aligned}$$ 
If $k_2\neq k_1, k_1-2, k_1+2,$ this implies by linear independence that all the coefficients $c_{a}^{**},c_{b}^{**},c_{c}^{**}$ are zero and thus $\omega=0.$ However if $k_2=k_1,$ only the coefficients $c_{b}^{**},c_{c}^{**}$ are zero by linear independence. So $\ds{\omega=\sum_{A_*,B_*}c_{a}^{**}~a_n\otimes A_*\w B_*}.$ Now $$0=[\omega,c_n]=2\sum_{A_*,B_*}c_{a}^{**}~c_n\otimes A_*\w B_*+\sum_{A_*,B_*}c_{a}^{**}~a_n\otimes A_*\w [ B_*,c_n].$$ This implies that  the coefficients $c_{a}^{**}$ are zero by linear independence. Hence $\omega=0.$\\
If $k_2=k_1-2,$ only the coefficients $c_{a}^{**},c_{c}^{**}$ are zero by linear independence. \\So $\ds{\omega=\sum_{A_*,B_*}c_{b}^{**}~b_n\otimes A_*\w B_*}.$ Now $$0=[\omega,c_n]=\sum_{A_*,B_*}c_{b}^{**}~a_n\otimes A_*\w B_*+\sum_{A_*,B_*}c_{b}^{**}~b_n\otimes A_*\w [ B_*,c_n].$$ This implies that  the coefficients $c_{b}^{**}$ are zero by linear independence. Hence $\omega=0.$\\
Similarly, if $k_2=k_1+2,$ only the coefficients $c_{a}^{**},c_{b}^{**}$ are zero by linear independence. Now the condition $[\omega,c_n]=0$ annihilates  the coefficients $c_{c}^{**}$  by linear independence. Hence $\omega=0.$
\end{proof}

\begin{lemma}\label{le4} The following are verctor space isomorphisms
$$[\so(n)\otimes\I_n^{\w k}]^{\bar{\h}_n}=\begin{cases}0 ,& \mbox{if} ~~k\neq 2,n-2\mbox{}\\ \langle \rho_n\rangle,& \mbox{if}~~k=2 \mbox{}\\ \langle\gamma_n\rangle,& \mbox{if}~~k=n-2 \mbox{}\end{cases}$$

where $$\rho_n=\sum_{1\leq i<j\leq n}X_{ij}\otimes y_i\w y_{n+j}-\sum_{1\leq i<j\leq n}X_{ij}\otimes y_j\w y_{n+i}$$ and 
$$
\begin{aligned}\gamma_n &=\sum_{{\substack{1\leq i<j\leq n \\ }}} X_{ij}\otimes( y_{1}\w\ldots\widehat{y_i}\ldots\w y_{n}\w y_{n+1}\w\ldots\widehat{y_{n+j}}\ldots\w y_{2n})\\&~~~~-\sum_{{\substack{1\leq i<j\leq n \\ }}} X_{ij}\otimes( y_{1}\w\ldots\widehat{y_j}\ldots\w y_{n}\w y_{n+1}\w\ldots\widehat{y_{n+i}}\ldots\w y_{2n})
\end{aligned}$$  
\end{lemma}

\begin{proof}

Clearly, $\ds{[\so(n)]^{\bar{\h}_n}=0}$ since $\ds{[X_{ij},X_{ik}]=X_{jk}\neq 0}.$ 
Also since $\so(n)$ is a subalgebra of $\bar{\h}_n,$ it follows by \cite[lemma 4.2]{BGR} that
$[\so(n)\otimes\w^k(\I_n^1)]^{\bar{\h}_n}$ is a submodule of $$[\so(n)\otimes\w^k(\I_n^1)]^{\so(n)}= \begin{cases}0 ,& \mbox{if} ~~k\neq 2,n-2\mbox{}\\ \big\langle \sum_{1\leq i<j\leq n}X_{ij}\otimes y_i\w y_j\big\rangle,& \mbox{if}~~k=2 \mbox{}\\ \big\langle\sum_{1\leq i<j\leq n}sgn(\sigma_{ij}) X_{ij}\otimes y_1\w\ldots \widehat{y_i}\ldots \widehat{y_j}\ldots\w y_n\big\rangle,& \mbox{if}~~k=n-2. \mbox{}\end{cases}$$ 
 However, $\ds{\big[ \sum_{1\leq i<j\leq n}X_{ij}\otimes y_i\w y_j,~b_n\big]= \sum_{1\leq i<j\leq n}X_{ij}\otimes y_{n+i}\w y_j+ \sum_{1\leq i<j\leq n}X_{ij}\otimes y_i\w y_{n+j}\neq 0}.$ \\ Also, one shows that $\ds{ \sum_{1\leq i<j\leq n}X_{ij}\otimes y_1\w\ldots \widehat{y_i}\ldots \widehat{y_j}\ldots\w y_n}\notin [\so(n)\otimes\w^{n-1}(\I_n^1)]^{\bar{\h}_n}.$ Hence \\$\ds{ [\so(n)\otimes\w^k(\I_n^1)]^{\bar{\h}_n}=0}$ for all $k.$ similarly, one shows that $\ds{ [\so(n)\otimes\w^k(\I_n^2)]^{\bar{\h}_n}=0}$ using $c_n.$
 Now let $r$ and $s$ with $\ds{1\leq r, s\leq n}$ and let $\ds{\omega\in [\so(n)\otimes(\I^1_n)^{\w r}\w(\I^2_n)^{\w s}]^{\bar{\h}_n}}.$ Then $$\ds{\omega=\sum_{1\leq i<j\leq n}c_{ij}^{**}~X_{ij}\otimes A_*\w B_*}$$ with $\ds{A_*\in(\I^1_n)^{\w r},~B_*\in(\I^2_n)^{\w s}}.$ If $r\neq s,$ we have $$0=[\omega,a_n]=(s-r)\sum_{{\substack{1\leq i<j\leq n \\A_*,B_* }}}c^{**}_{ij}~X_{ij}\otimes A_*\w B_*$$ as $[A_*,a_n]=-r A_*$ and $[B_*,a_n]=s B_*.$ So all the $c^{**}_{ij}s$ are zero by linear independence. So $\omega=0.$ For the case $r=s,$ first notice that since $\so(n)$ is a Lie subalgebra of $\bar{\h}_n,$ it follows that $\ds{[\so(n)\otimes\w^k(\I_n)]^{\bar{\h}}\subseteq [\so(n)\otimes\w^k(\I_n)]^{\so(n)}}$. Now following again the proof of  \cite[lemma 4.2]{BGR}, we have  that $$dim[\so(n)\otimes\I_n^{\w k}]^{\so(n)}=dim Hom_{\so(n)}(\I_n^{\w 2},~\I_n^{\w k})=\begin{cases}1 ,& \mbox{if} ~~k= 2,2n-2\mbox{}\\0,& \mbox{else} \mbox{}\end{cases}.$$
So two cases remain here for  possible $\bar{\h}_n$-invariants: Firstly, if $r=s=1$ we have\\ 
 $\ds{\omega=\sum_{{\substack{1\leq i<j\leq n \\1\leq k,l\leq n }}}c_{ij}^{kl}X_{ij}\otimes x_k\frac{\partial}{\partial x^{n+1}}\w x_l\frac{\partial}{\partial x^{n+2}}}.~$ However  since $[\omega, b_n]=[\omega, c_n]=0,$ it follows by linear independence that all coefficients $c_{ij}^{kl},~k\neq l$ are zero except for $(i,j)=(k,l)$ in which case  successive choices of $X_{ij}\in\so(n)$ with $[X_{ij},\omega]=0$ imply by linear independence that $c_{ij}^{ij}=-c_{ji}^{ji}=c$ for some constant c.Thus $$\omega=c\sum_{1\leq i<j\leq n}X_{ij}\otimes y_i\w y_{n+j}-\sum_{1\leq i<j\leq n}X_{ij}\otimes y_j\w y_{n+i}.$$
Secondly, if $r=s=n-1,$   the proof that $\omega=c\gamma_n$ for some constant $c$ is similar to the first case.
\end{proof}
\begin{lemma}\label{le5}  For all integer $k,$
$$[\I_n\otimes\I_n^{\w k}]^{\bar{\h}_n}=0.$$
\end{lemma}
\begin{proof}
Write $\I_n=\I_n^1\oplus\I_n^2$ as in the proof of  lemma \ref{lem2}. Note that since $\so(n)$ is a Lie subalgebra of $\bar{\h}_n,$ it follows by \cite[lemma 4.3]{BGR} that  $\ds{[\I_n^i\otimes\w^k(\I^i_n)]^{\bar{\h}_n}\subseteq [\I_n^i\otimes\w^k(\I^i_n)]^{\so(n)}=0}$ for all $k\neq 1, n-1,$ for $i=1,2.$
Now  by \cite[lemma 4.4]{BGR}, $$[\I_n^1\otimes\I^1_n]^{\bar{\h}_n}\subseteq [\I_n^1\otimes\I^1_n]^{\so(n)}=\big\langle\sum_{i=1}^{n}y_i\otimes y_i\big\rangle$$ and  by \cite[lemma 4.5]{BGR}, $$[\I_n^1\otimes\w^{n-1}(\I^1_n)]^{\bar{\h}_n}\subseteq [\I_n^1\otimes\w^{n-1}(\I^1_n)]^{\so(n)}=\big\langle\sum_{m=1}^{n}(-1)^{m-1} y_m\otimes y_1\w y_2\ldots\widehat{y_m}\ldots\w y_{n}\big\rangle.$$ However $$[\sum_{i=1}^{n}y_i\otimes y_i, b_n]=-\sum_{i=1}^{n}\big(y_i\otimes y_{n+i}+y_{n+i}\otimes y_i\big)\neq 0$$ and\\ $[\sum_{m=1}^{n}(-1)^{m-1} y_m\otimes y_1\w \ldots\widehat{y_m}\ldots\w y_{n}, b_n]=-\sum_{m=1}^{n}(-1)^{m} y_{n+m}\otimes y_1\w \ldots\widehat{y_m}\ldots\w y_{n}\\~~~~~~~~~~~~~~~~~~~~~~~~~~~~~~~+\sum_{i=1}^{n}\sum_{m=1}^{n}(-1)^{m-1} y_m\otimes y_1\w\ldots\w y_{n+i}\w \ldots\widehat{y_m}\ldots\w y_{n}\neq 0.$ So $$[\I_n^1\otimes\w^{n-1}(\I^1_n)]^{\bar{\h}_n}=0  ~~\mbox{and}~~[\I_n^1\otimes\I^1_n]^{\bar{\h}_n}=0.$$Similarly, $$[\I_n^2\otimes\w^{n-1}(\I^2_n)]^{\bar{\h}_n}=0  ~~\mbox{and}~~[\I_n^2\otimes\I^2_n]^{\bar{\h}_n}=0.$$

Now in general, let $k_1$ and $k_2$ with $1\leq k_1,k_2\leq n$ and let $\ds{\omega\in [\I_n\otimes(\I^1_n)^{\w k_1}\w(\I^1_n)^{\w k_2}]^{\bar{\h}_n}}.$ Then $\ds{\omega=\sum_{{\substack{1\leq i\leq n \\A_*,B_* }}}c^{**}_i~y_i\otimes A_*\w B_*+\sum_{{\substack{1\leq i\leq n \\A_*,B_* }}}c^{**}_{n+i}~y_{n+i}\otimes A_*\w B_*}$ \\with $A_*\in(\I^1_n)^{\w k_1},~B_*\in(\I^2_n)^{\w k_2}$ and $c_i^{**},~ c_{n+i}^{**}$ real coefficients. We have 
$$
\begin{aligned}
0=[\omega,a_n]=&(k_2-k_1-1)\sum_{A_*,B_*}c_{i}^{**}~y_i\otimes A_*\w B_*+(k_2-k_1+1)\sum_{A_*,B_*}c_{n+i}^{**}~y_{n+i}\otimes A_*\w B_*.
\end{aligned}$$ Two cases occur: Firstly, if $k_2=k_1+1,$ all the coefficients $c_{n+i}^{**}$ are zero. Now since
$$0=[\omega, b_n]=\sum_{{\substack{1\leq i\leq n \\A_*,B_* }}}c^{**}_i~y_{n+i}\otimes A_*\w B_*+\sum_{{\substack{1\leq i\leq n \\A_*,B_* }}}c_{i}^{**}~y_{i}\otimes[ A_*,b_n]\w B_*,$$ it follows that all the coefficients $c^{**}_i$ are zero by linear independence. So $\omega=0.$ Secondly, if $k_2=k_1-1,$ all the coefficients $c_{i}^{**}$ are zero. Similarly, the condition
$0=[\omega, c_n]$ annihilates  all the coefficients $c^{**}_{n+i}$  by linear independence. So $\omega=0.$ 
\end{proof}
As a consequence of these lemmas 
we have the following:


\begin{theorem}\label{lb} There are graded vector space isomorphisms
$$H^{Lie}_{*}(\shn_n;~\mathbb{R})\cong H^{Lie}_{*}(\sln(2);~\mathbb{R})\otimes H^{Lie}_{*}(\so(n);~\mathbb{R})\otimes\big(\mathbb{R}\oplus\langle\zeta_n\rangle\oplus\langle\alpha_n\rangle\big),$$  and
$$ HL_*(\shn;\mathbb{R})\cong\big(\mathbb{R}\oplus\langle \tilde{\zeta}_{n}\rangle\oplus\left\langle \tilde{\alpha}_{n}\right\rangle\big)\otimes T^*(\tilde{\gamma}_n),$$ where 
$\ds{\tilde{\alpha}_n=\frac{1}{(2n)!}\sum_{\sigma\in S_{2n}}\mbox{sgn} (\sigma)y_{\sigma(1)}\otimes\ldots\otimes y_{\sigma(n)}\otimes y_{\sigma(n+1)}\otimes\ldots\otimes y_{\sigma(2n)}}$ is the antisymmetrization of $\alpha_n,$ 
$$
\begin{aligned}\tilde{\zeta}_n &=\frac{1}{(2n)!}\sum_{{\substack{\\\sigma\in S_{2n-2} }}}\mbox{sgn} (\sigma)y_{\sigma(1)}\otimes\ldots\widehat{y_{\sigma(i)}}\ldots\otimes y_{\sigma(n)}\otimes y_{\sigma(n+1)}\otimes\ldots\widehat{y_{\sigma(n+i)}}\ldots\otimes y_{\sigma(2n)}
\end{aligned}$$ is the antisymmetrization of $\zeta_n$  and $\tilde{\gamma}_n$ is the $\bar{\h}_n$-invariant cycle in $\shn_n^{\otimes(n-1)}$ representing $\gamma_n.$ 

\end{theorem}

\begin{proof}
The first isomorphism follows by lemma \ref{lem0},  lemma \ref{lem2} and the K\"unneth formula for Lie algebra homology. Note that the $\bar{\h}_n$- invariants $\beta_n$ is zero in  $H^{Lie}_{*}(\shn_n;~\mathbb{R})$ since $$\ds{d(\bar{\rho}_n)=-2(n-1)\beta_n}$$ where $d$ is the Chevalley-Eilenberg boundary map and  $$\ds{\bar{\rho}_n=\sum_{1\leq i<j\leq n}X_{ij}\w y_i\w y_{n+j}-\sum_{1\leq i<j\leq n}X_{ij}\w y_j\w y_{n+i}}.$$ For the second isomorphism, notice that $\gamma_n$ and $\rho_n$ are the only  $\bar{\h}_n$- invariants in $\shn_n\otimes\I_n^{\w *}.$  But  $\rho_n$  is not a cycle in $H_*^{Lie}(\I_n; ~\shn_n)^{\bar{\h}_n}$ as it maps to $$\ds{d(\rho_n)=-2(n-1)\sum_{i=1}^n y_i\otimes y_{n+i}\neq 0}$$ by  the Loday boundary map $d.$  So the kernel $K_*$ of the composition $$\pi_*\circ j_*:H^{Lie}_*(\I_n; ~\shn_n)^{\bar{\h}_n}\rightarrow H^{Lie}_*(\shn_n; ~\shn_n)\rightarrow H^{Lie}_{*+1}(\shn_n;~\mathbb{R})$$  is $K_*=\langle \gamma_{n}\rangle.$ Note  that  ${\pi_*\circ j_*\circ d(\rho_n)=-2(n-1)\beta_n\neq 0}$ and  ${\pi_*\circ j_*(\gamma_n)=0}$ in $H^{Lie}_{*+1}(\shn_n;\mathbb{R})$
 Now since $\bar{\h}_n$ is a semisimple Lie algebra, we have by Lodder's structure theorem  \cite[Lemma 3.6]{JL}  that $$HL_*(\shn_n;~\mathbb{R})\cong [\wedge^*(\I_n)]^{\bar{\h}_n}\otimes T(K_*)\cong\big(\mathbb{R}\oplus\langle \tilde{\zeta}_{n}\rangle\oplus\left\langle \tilde{\alpha}_{n}\right\rangle\big)\otimes T^*(\tilde{\gamma}_n).$$ (The construction of $\tilde{\gamma}_n$ is similar to \cite[lemma 4.6]{BGR}).
\end{proof}

\section{Leibniz homology of the  full Galilei algebra}

We now turn back our attention to the  full Galilei algebra  $\widetilde{\g}_n.$ From the matrix (M.1), we can construct the Lie algebra isomorphisms  $$\ds{\gln(2; \mathbb{R})\cong Span\{a_{n},b_{n},c_{n},d_n\}~~\mbox{and}~~~\widetilde{\g}_n\cong \shn_n\oplus \langle d_n\rangle}$$ where
 $$\ds{d_n:=x_{n+1}\frac{\partial}{\partial x^{n+1}}+x_{n+2}\frac{\partial}{\partial x^{n+2}}},$$ and
in addition to the brackets in the Lie algebra $\shn_n,$ we have the following brackets
$$\ds{[X_{ij},d_{n}]=0,~~~[a_{n},d_{n}]=0,~~~[b_{n},d_{n}]=0,~~~~[c_{n},d_{n}]=0.}$$
$$\ds{[d_{n},x_i\frac{\partial}{\partial x^{n+1}}]=-x_i\frac{\partial}{\partial x^{n+1}} ~~~[d_{n},x_i\frac{\partial}{\partial x^{n+2}}]=-x_i\frac{\partial}{\partial x^{n+2}}}.$$


\begin{corollary} There is  a graded vector space isomorphism
$$ HL_*(\widetilde{\g}_n;~\mathbb{R})\cong\big(\big(\mathbb{R}\oplus\langle \tilde{\zeta}_{n}\rangle\oplus\left\langle \tilde{\alpha}_{n}\right\rangle\big)\otimes T^*(\tilde{\gamma}_n)\big)*T^*(\mathbb{R})$$ where $*$ is the non-commutative tensor product of $\mathbb{N}$-graded modules.

\end{corollary}

\begin{proof}

Since $\widetilde{\g}_n\cong \shn_n\oplus \langle d_n\rangle$ and $ H^{Lie}_{*}(\langle d_n\rangle;~\mathbb{R})\cong T^*(\mathbb{R}),$ the result follows by \cite[theorem 3]{LK}  and theorem \ref{lb}.
\end{proof}



\begin{thebibliography}{aaaa}


\bibitem{BGR} Biyogmam, G. R., \textit{On the Leibniz (Co)homology of the Lie Algebra of
the Euclidean Group,}  Journal of Pure and Applied Algebra, \textbf{215} (2011),
1889 -1901.
\bibitem{RGB} Biyogmam, G. R., \textit{ Leibniz Homology of the Affine Indefinite Orthogonal Lie Algebra,} 2013. arxiv:1301.0659v1.

\bibitem{CE} Chevalley, C., Eilenberg, S.,  \textit{Cohomology theory of Lie groups and Lie algebras}, Trans. Amer. Math. Soc, \textbf{63}, 1 (1948), 85-124.  

\bibitem{H} Hilton, P. J., Stammbach, U.,  \textit{A course in homological algebra}, Springer-Verlag, New York, 1971. 

\bibitem{IK} Ito, K., Nihon, S., \textit{Encyclopedic Dictionary of Mathematics}, Cambridge, Mass: MIT Press, 1987. 




\bibitem{KA} Kostrikin, A. I., Manin, I.,`` Linear Algebra and Geometry, Algebra, logic, and applications'', Gordon and Breach Science Publishers, \textbf{Vol. 1,} New York, 1989. 



\bibitem{L} Loday, J.-L., ``Cyclic Homology, Springer-Verlag'', Berlin, Heidelberg, New York, 1992.

\bibitem{LK} Loday, J.-L.,  \textit{K\"unneth-style formula for the homology of Leibniz algebras}, \textit{Mathematische Zeitschrift,} \textbf{221} (1996), 41-47.
\bibitem{JLM} Lodder, J. M.,  \textit{Lie algebras of Hamiltonian vector fields and symplectic manifold}, \textit{Journal of Lie Theory}, \textbf{18}, 4 (2008),  897--914.
\bibitem{JL} Lodder, J., \textit{A Structure Theorem for Leibniz Cohomology}, Journal of Algebra,  \textbf{355}, 1 (2012), 93 - 110.
\bibitem{TP} Pirasvili, T., \textit{On Leibniz Homology}, Annales de l'institut Fourrier, Grenoble \textbf{44}, 2 (1994), 401 - 411.
















\end{thebibliography}
\end{document}